\documentclass[12pt,reqno]{amsart}

\usepackage[a4paper,left=30mm,right=30mm,top=30mm,bottom=30mm,marginpar=25mm]{geometry}

\usepackage{color}
\usepackage{comment}
\usepackage{amsmath}
\usepackage{amssymb}
\usepackage{amsthm}
\usepackage[latin1]{inputenc}
\usepackage{eurosym}
\usepackage[dvips]{graphics}
\usepackage{graphicx}
\usepackage{epsfig}
\usepackage{dsfont}
\usepackage[displaymath,mathlines]{lineno}
\usepackage{mathtools}
\usepackage[nocompress, space]{cite}

\allowdisplaybreaks

\usepackage{enumitem}

\usepackage{ifthen}

\makeindex

\newcommand{\Rr}{{\mathbb{R}}}

\newcommand{\Nn}{{\mathbb{N}}}

\newcommand{\Tt}{{\mathbb{T}}}

\newcommand{\epsi}{\varepsilon}

\def\leq{\leqslant}
\def\geq{\geqslant}

\numberwithin{equation}{section}

\newtheoremstyle{thmlemcorr}{10pt}{10pt}{\itshape}{}{\bfseries}{.}{10pt}{{\thmname{#1}\thmnumber{
#2}\thmnote{ (#3)}}}
\newtheoremstyle{thmlemcorr*}{10pt}{10pt}{\itshape}{}{\bfseries}{.}\newline{{\thmname{#1}\thmnumber{
#2}\thmnote{ (#3)}}}
\newtheoremstyle{defi}{10pt}{10pt}{\itshape}{}{\bfseries}{.}{10pt}{{\thmname{#1}\thmnumber{
#2}\thmnote{ (#3)}}}
\newtheoremstyle{remexample}{10pt}{10pt}{}{}{\bfseries}{.}{10pt}{{\thmname{#1}\thmnumber{
#2}\thmnote{ (#3)}}}
\newtheoremstyle{ass}{10pt}{10pt}{}{}{\bfseries}{.}{10pt}{{\thmname{#1}\thmnumber{
A#2}\thmnote{ (#3)}}}

\theoremstyle{thmlemcorr}
\newtheorem{theorem}{Theorem}
\numberwithin{theorem}{section}

\theoremstyle{thmlemcorr*}
\newtheorem{theorem*}{Theorem}
\newtheorem{lemma*}[theorem]{Lemma}
\newtheorem{corollary*}[theorem]{Corollary}
\newtheorem{proposition*}[theorem]{Proposition}
\newtheorem{problem*}[theorem]{Problem}
\newtheorem{conjecture*}[theorem]{Conjecture}

\theoremstyle{defi}
\newtheorem{definition}[theorem]{Definition}

\theoremstyle{remexample}
\newtheorem{remark}[theorem]{Remark}

\newtheorem{teo}[theorem]{Theorem}
\newtheorem{lem}[theorem]{Lemma}
\newtheorem{pro}[theorem]{Proposition}
\newtheorem{cor}[theorem]{Corollary}

\theoremstyle{ass}

\begin{document}

\title[Vanishing discount problem for infinite system]{Remarks on the vanishing discount problem for infinite systems of Hamilton-Jacobi-Bellman equations}

\author{Kengo Terai}
\address[K. Terai]{
        Graduate School of Mathematical Sciences, The University of Tokyo, 3-8-1 Komaba, Meguro-ku, Tokyo, 153-8914, Japan.}
\email{terai@ms.u-tokyo.ac.jp}

\thanks{
        This work was supported by Grant-in-Aid for JSPS Fellows Grant number 20J10824. 
}
\date{\today}

\begin{abstract}
This paper is concerned with the asymptotic analysis of infinite systems of weakly coupled stationary Hamilton-Jacobi-Bellman equations as the discount factor tends to zero. With a specific Hamiltonian, we show the convergence of the solution and  prove the solvability of the corresponding ergodic problem.
\end{abstract}

\maketitle
\section{introduction}
In this paper, we consider the following functional partial differential equation,
\begin{equation}\label{DP}
\alpha v^\alpha(x,\xi)+H(x,Dv^\alpha(x,\xi),\xi)+\int_Ik(\xi,\eta)(v^\alpha(x,\xi)-v^\alpha(x,\eta))\mbox{ }d\eta=0 \quad\rm{in}\  \Tt^d\times I,
\end{equation}
where, $I \subset \Rr$ is a given finite interval, $\Tt^d$ is the $d$-dimensional flat torus, and  $\alpha$ is a given positive number. Hamiltonian $H: \Tt^d \times \Rr^d \times I \to \Rr$ is a given function.
 A function $k:I \times I \to \Rr$ is positive and bounded, in particular, we assume that there exist positive constants $k_0<k_1$ such that 
\begin{equation*}
k_0\leq k(\xi,\eta)\leq k_1 \quad  (\xi,\eta \in I).
\end{equation*}
For simplicity, we assume that $|I|=1$, where $|\cdot|$ denotes the Lebesgue measure. In \eqref{DP}, $v^\alpha: \Tt^d \times I \to \Rr$ is an unknown function.

Here, we interpret $\xi \in I$ as the component of PDE systems.
Then, the functional partial differential equation \eqref{DP} can be regarded as an uncountable infinite system of discounted Hamilton-Jacobi-Bellman (HJB) equations. 
Weakly coupled HJB equations arise, for example, in the literature of optimal control problems with random switching costs, which are governed by specific Markov chains with state space $I$; see \cite{FS} and \cite{YZ}, for instance.
When $I=\{1,2,\cdots,m\}$ for a fixed $m\in \Nn$, the well-posedness of viscosity solutions was established in \cite{Enger} and \cite{Koike}.
When $I \subset \Rr$ is a given finite interval, in \cite{Ishiishimano}, they worked with viscosity solutions and addressed the asymptotic behavior of a different type using the perturbed test function method. In \cite{Shimano}, the author studied simultaneous effects of homogenization and penalization for infinite systems.   


Main interest of this paper is the {\it vanishing discount problem}; that is, how $v^\alpha$, the solution to \eqref{DP}, behaves as the {\it discount factor} $\alpha$ tends to zero. In view of the results of single equations or finite systems, we expect that the limit problem associated with \eqref{DP} is
\begin{equation}\label{EP}
H(x,Dv(x,\xi),\xi)+\int_Ik(\xi,\eta)(v(x,\xi)-v(x,\eta))\mbox{ }d\eta=c \quad\rm{in}\  \Tt^d\times I .
\end{equation}
The unknown is  a pair of $v:\Tt^d \times I \to \Rr$ and $c\in \Rr$, whose existence is not well-established in a general setting. Finally, via this discount approximation, we address the solvability of \eqref{EP}.


Here, we review some results of the vanishing discount problems briefly. 
In the case of single equations, for coercive Hamiltonian, the results in \cite{LPV} showed the existence of viscosity solutions to the ergodic problem 
\begin{equation}\label{EE}
H(x,Dv)=c \quad\rm{in}\  \Tt^d,
\end{equation}
which plays an important role in many contexts of PDE, for example, when we investigate the large time behavior or homogenization of Hamilton-Jacobi equation.
To prove the existence, it is natural to consider the following discount approximation,
\begin{equation}\label{DD}
\alpha v^\alpha+H(x,Dv^\alpha)=0 \quad\rm{in}\  \Tt^d.
\end{equation}
Because of the term $\alpha v^\alpha$, \eqref{DD} has the unique viscosity solution $v^\alpha$. Then,  $\alpha v^\alpha$ is uniformly bounded and $v^\alpha$ is equi-Lipschitz with respect to $\alpha >0$. Thus, through subsequences, $v^\alpha -\min_{\Tt^d} v^\alpha $ and $\alpha v^\alpha$ uniformly converges to some $v\in C(\Tt^d)$ and $-c\in \Rr$, respectively. Moreover, by the stability of viscosity solutions, $(v,c)$ solves the ergodic problem \eqref{EE}.
On the other hand, due to nonuniqueness of viscosity solutions to \eqref{EE}, it was not clear whether the convergence of $v^\alpha -\min_{\Tt^d} v^\alpha $ holds for the whole sequence or not.
From a view point of weak KAM theory, several researchers have investigated this asymptotic behavior. Firstly, the convergence of the whole sequence was established in \cite{davifathi}. The results in \cite{MT5} and \cite{IMT} proved the convergence for a degenerated second-order equation. In the case of finite systems, in \cite{davizavi} and \cite{Ishii2,Ishii3}, they showed  convergence results. On the other hand, in \cite{Ziliotto}, there is an example with non-convex Hamiltonian for which the convergence does not hold for the whole sequence. In \cite{Ishii4}, using the idea of \cite{Ziliotto}, the author showed a counterexample of the convergence for a system.  

 In the case of uncountable infinite systems, it is worth emphasizing that we need to take into account possibly the lack of the equi-continuity of $v^\alpha$, particularly with respect to the parameter $\xi$. Thus, in general, we do not know whether $v^\alpha -\min_{\Tt^d} v^\alpha $ converges on $\Tt^d \times I$ even along subsequences (see Remark \ref{rem}). 
This makes the problem much harder, and therefore we start with a specific Hamiltonian as below in \eqref{form}.
 This kind of Hamiltonian was studied in \cite{Namah} for single equations in the context of the large time behavior of Hamilton-Jacobi equations. Later, in \cite{cami2} and \cite{MT2}, the authors worked with this Hamiltonian when they investigated the large time behavior for finite systems. In this paper, we use their ideas to study the vanishing discount problem for infinite systems.

Here, we recall the definition of viscosity solutions for infinite systems  (see \cite{Ishiishimano} and \cite{Shimano}, for instance).
\begin{definition}
We denote by $\mathcal{U}^+(\Tt^d \times I)$ the set of functions $u$ on $\Tt^d \times I$ such that for all $x\in \Tt^d$ the function $u(x,\cdot)$ is Borel measurable and integrable in $I$ with respect to the given kernel $k$ and for each $\xi \in I$ the function $u(\cdot,\xi)$ is upper semi-continuous in $\Tt^d$. We set $\mathcal{U}^-(\Tt^d \times I):=-\mathcal{U}^+(\Tt^d\times I)$.
\begin{enumerate}
\item $u \in \mathcal{U}^+(\Tt^d \times I)$ is a viscosity subsolution to \eqref{DP} if whenever $\phi \in C^1(\Tt^d)$, $\xi \in I$ and $u(\cdot,\xi)-\phi$ attains its local maximum at $\hat x$, then
\begin{equation*}
\alpha u(\hat x,\xi)+H(\hat x,D \phi(\hat x),\xi)+\int_Ik(\xi,\eta)(u(\hat x,\xi)-u(\hat x,\eta))\mbox{ }d\eta \leq 0.
\end{equation*} 

\item $u\in \mathcal{U}^-(\Tt^d \times I)$ is a viscosity supersolution to \eqref{DP} if whenever $\phi \in C^1(\Tt^d)$, $\xi \in I$ and $u(\cdot,\xi)-\phi$ attains its local minimum at $\hat x$, then
\begin{equation*}
\alpha u(\hat x,\xi)+H(\hat x,D \phi(\hat x),\xi)+\int_Ik(\xi,\eta)(u(\hat x,\xi)-u(\hat x,\eta))\mbox{ }d\eta \geq 0.
\end{equation*} 

\item $u\in \mathcal{U}^+(\Tt^d \times I) \cap \: \mathcal{U}^-(\Tt^d \times I)$ is a viscosity solution to \eqref{DP} if $u$ is both a viscosity subsolution and a viscosity supersolution to \eqref{DP}.
\end{enumerate}
\end{definition}

Let $\mathcal{B}(I)$ be the space of all Borel functions on $I$.
Let $C(\Tt^d) \otimes \mathcal{B}(I)$ denote the set of functions  $g$ on $\Tt^d \times I$ such that for each $x\in \Tt^d$ the function $g(x,\cdot)$ is Borel measurable in $I$ and for each $\xi \in I$ the function $g(\cdot,\xi)$ is continuous on $\Tt^d$.

In \cite{Ishiishimano}, the authors proved a comparison principle and Perron's method for a infinite system. By a similar argument, we can prove that there exists  $v^\alpha \in C(\Tt^d) \otimes \mathcal{B}(I)$, which is the unique viscosity solution to \eqref{DP}. 

Before stating the main result, we introduce some assumptions. When we regard \eqref{DP} as an infinite system, it is natural that we do not assume that Hamiltonian $H$ and the kernel $k$ is continuous with respect to the parameter $\xi \in I$ and $(\xi,\eta)\in I\times I$, respectively.
Here, we consider the Hamiltonian $H\in C(\Tt^d\times \Rr^d)\otimes \mathcal{B}(I)$ of the form 
\begin{equation}\label{form}
H(x,p,\xi)=F(x,p,\xi)-f(x,\xi),
\end{equation}
where $F\in C(\Tt^d \times\Rr^d) \otimes \mathcal{B}(I)$ and $f\in C(\Tt^d) \otimes \mathcal{B}(I)$. We assume that $f$ is bounded on $\Tt^d \times I$, and that $F(\cdot,p,\cdot)$ is bounded on $\Tt^d\times I$ for each $p\in \Rr^d$. We work under the following assumptions.

\begin{enumerate}[label=(A\arabic*)]
\item There exist constants $C_1, C_2 >0$ and $m>1$ such that for all $(x,p,\xi)\in \Tt^d \times \Rr^d \times I$,
\[C_1 |p|^m -C_2 \leq F(x,p,\xi).\]
\label{Coercive}
\item $f$ is nonnegative and
\begin{equation}\label{A}
\mathcal A:=\{ x\in \Tt^d \:\:|\:\: f(x,\xi)=0\quad \mbox{for all } \xi \in I\} \neq \emptyset.
\end{equation}
\label{NR}
\item For each $(x,\xi)\in \Tt^d \times I$, $p \mapsto F(x,p,\xi)$ is convex.
\label{convex}
\item 
$F$ is nonnegative and $F(x,0,\xi)=0$, for all $(x,\xi) \in \Tt^d \times I$.
\label{positive}
\item 
There exists a modulus of continuity $\omega$ such that
\[|f(x,\xi)-f(y,\xi)|\leq \omega(|x-y|),\]
for all $x,y \in \Tt^d$ and $\xi \in I$.\\
For each $R>0$, there exists  a modulus of continuity $\omega_R$ such that
\[ |F(x,p,\xi)-F(y,q,\xi)|\leq \omega_R(|x-y|+|p-q|), \]
for all $x,y \in \Tt^d$, $p,q\in B(0,R)$ and $\xi \in I$, where $B(0,R)$ is the closed ball with radius $R$ centered at the origin.
\label{growth1}
\end{enumerate}
\setcounter{theorem}{0}

Here is our main result of this paper.
\begin{teo}\label{result1}
Suppose that Assumptions \ref{Coercive}-\ref{growth1} hold. Let $v^\alpha \in C(\Tt^d)\otimes \mathcal{B}(I)$ be the viscosity solution to \eqref{DP}. Then, there exists $v\in C(\Tt^d)\otimes \mathcal{B}(I)$, which is a viscosity solution to \eqref{EP} with $c=0$, satisfying that for all $\xi \in I$, $v^\alpha(\cdot,\xi) \to v(\cdot,\xi)$ uniformly on $\Tt^d$ as $\alpha \to 0$.
\end{teo}
To show this convergence result, we use an idea from \cite{Namah}. In the case of both single equations and finite systems (see \cite{cami2} and \cite{MT2}), a comparison principle in terms of $\mathcal{A}$ (or $\mathcal{A}\times I$) plays an important role to study asymptotic problems. 
In our setting, because we study under the condition that  $F\in C(\Tt^d\times \Rr^d)\otimes \mathcal{B}(I)$, $k$ is measurable on $I\times I$, and in particular $f\in C(\Tt^d)\otimes \mathcal{B}(I)$,
we set
\begin{equation}\label{tildeA}
\mathcal{\tilde A}:=\{x\in \Tt^d \:\:|\:\: f(x,\xi)=0 \quad \mathrm{a.e.}\:\: \xi \in I\},
\end{equation}
and
\begin{equation}\label{uniquenessset}
\mathcal{Z}:=\{(x,\xi)\in \mathcal{\tilde A} \times I \:\:|\:\: f(x,\xi)=0 \}. 
\end{equation}
It is clear that $\mathcal{A} \subset \mathcal{\tilde A}$ and $\mathcal{A}\times I \subset \mathcal{Z}$. Note that  if $f$ is continuous on $\Tt^d \times I$, it holds  $\mathcal{A} = \mathcal{\tilde A}$ and  $\mathcal{A}\times I = \mathcal{Z}$. 
We emphasize that $\mathcal{Z}$ is a uniqueness set for \eqref{EP} as proved in Theorem \ref{SCR}.
Although $\mathcal{Z}$ is bigger than $\mathcal{A} \times I$, we have nice controls of the values of $v^\alpha$ on $\mathcal{Z}$ thanks to Proposition $4.1$. They enable us to show that the convergence of $v^\alpha$ holds for the whole sequence.

This paper is organized as follows. In Section \ref{gamma}, we review some basic properties of coupling term of weakly coupled systems. In Section \ref{discount}, we get estimates for the solution to \eqref{DP} and stability results of viscosity solutions.  Finally, in Section \ref{conv}, we give a comparison principle for ergodic problem \eqref{EP} and prove Theorem \ref{result1}.

\section{Preliminaries}
\label{gamma}

In this paper, we often use the following notation:
\[ \Theta v(x,\xi) := \int_Ik(\xi,\eta)(v(x,\xi)-v(x,\eta))\mbox{ }d\eta .\]
Here, we review some properties of coupling term $\theta$, which have been studied in \cite{Ishiishimano} and \cite{Shimano}. To keep our presentation self-contained, we summarize some of them.

Define the compact linear operator $ \Lambda: {L}^2 (I) \to {L}^2(I)$ by
\[  \Lambda f(\xi):= \int_I \frac{k(\xi,\eta)}{\overline k(\xi)}f(\eta)\mbox{ }d\eta, \]
where we define $\overline k(\xi)$ by
\[\overline k(\xi):=\int_I k(\xi,\eta)\mbox{ }d\eta. \]
Then, $\Lambda$ has $1 \in \Rr$ as its eigenvalue and $\bold{1}(\xi)\equiv1 \in L^2(I)$ as a corresponding eigenvector.
Then, the following Lemma holds. 

\begin{lem} Let $r( \Lambda)$ be the spectral radius of $\Lambda$. Then, $r(\Lambda)=1$.
\end{lem}

\begin{proof}
Let $\mu \in \mathbb{C}$ be an eigenvalue of $\Lambda$ with $|\mu|\geq1$ and take $\hat \phi \in L^2(I)$ as a corresponding eigenvector, that is, $\Lambda   \hat \phi=\mu  \hat \phi$ in the sense of $L^2(I)$. Define the function $\phi: I\to \Rr$ by
\[\phi(\xi)=\mu^{-1}\int_{I} \frac{k(\xi,\eta)}{\overline k(\xi)}\hat \phi(\eta) \mbox{ }d\eta. \]
Then, we obtain $\mu \phi(\xi)=\Lambda \phi(\xi)$ for all $\xi \in I$.

Set $M:=\sup_{I}|\phi|$. Here, we show that $|\phi(\xi)|=M$ for almost all $\xi \in I$. Take $\epsi>0$ and $\xi \in I$ satisfying $|\phi(\xi)|>M-\epsi$. Let $\delta>\epsi$ and set $B_\delta=\{\xi \in I \:\: |\:\: M-\delta\geq |\phi(\xi)|\}$.
Note that we have
\begin{equation*}
M-\epsi<|\phi(\xi)|\leq |\mu \phi(\xi)|\leq \int_I \frac{k(\xi,\eta)}{\overline k(\xi)} |\phi(\eta)|\mbox{ }d\eta.
\end{equation*}
 Then, we observe that
\begin{align*}
0
&\leq \int_I \frac{k(\xi,\eta)}{\overline k(\xi)}(|\phi(\eta)|-M+\epsi)\mbox{ }d\eta \\
&\leq \int_{B_\delta} \frac{k(\xi,\eta)}{\overline k(\xi)}(\epsi-\delta)\mbox{ }d\eta+\int_{I}\frac{k(\xi,\eta)}{\overline k(\xi)} \epsi \mbox{ }d\eta \leq -(\delta-\epsi)\frac{k_0}{k_1}|B_\delta|+\epsi .
\end{align*}
As $\epsi \to 0$, we get $|B_\delta|=0$ for all $\delta>0$, which implies $|\phi(\xi)|=M$ for almost all $\xi \in I$. 
Thus, for almost all $\xi \in I$, it holds that
\begin{equation*}
|\mu|M=|\mu||\phi(\xi)|\leq \int_{I} \frac{k(\xi,\eta)}{\overline k(\xi)} |\phi(\eta)|\mbox{ }d\eta=M,
\end{equation*}
which implies $|\mu|=1$.
\end{proof}

Let $\Lambda^*$ be the adjoint operator of $\Lambda$ given by
\[  \Lambda^* f(\eta):= \int_I \frac{k(\xi,\eta)}{\overline k(\xi)}f(\xi)\mbox{ }d\xi. \]
 Note that $r(\Lambda^*)=r(\Lambda)=1$.
By the Perron-Frobenius theorem or the Krein-Rutman theorem, 
there exists a positive $ \hat \gamma \in L^2(I)$ such that $\Lambda^* \hat \gamma=\hat \gamma$ in the sense of $L^2(I)$.
Define $\gamma: I\to \Rr $ by
\begin{equation*}
\gamma(\eta)=\int_I \frac{k(\xi,\eta)}{\overline k(\xi)}  \hat \gamma(\xi) \mbox{ }d\xi.
\end{equation*}
Then, $\gamma$ is a positive and bounded function satisfying that 
\begin{equation}\label{gamma1}
\int_I \frac{k(\xi,\eta)}{\overline k(\xi)} \gamma(\xi) \mbox{ }d\xi=\gamma(\eta)\quad (\eta \in I).
\end{equation}
Moreover, 
 we may assume $\|\gamma\|_{L^1(I)}=1$, and from \eqref{gamma1}, we have
\begin{equation}\label{gammaesti}
\frac{k_0}{k_1}\leq \gamma(\xi)\leq \frac{k_1}{k_0} \quad (\xi \in I).
\end{equation}
Here, we check the following property.
\begin{pro} Let $f \in L^1(I)$. Then,
\begin{equation}\label{couple}
\int_I \frac{\gamma(\xi)}{\overline k(\xi)}\Theta f(\xi) \mbox{ }d\xi=0.
\end{equation}
\end{pro}

\begin{proof}Using \eqref{gamma1} and Fubini's theorem, we obtain
\begin{align*}
\int_I \frac{\gamma(\xi)}{\overline k(\xi)} \Theta f(\xi) \mbox{ }d\xi
&=\int_I \frac{\gamma(\xi)}{\overline k(\xi)}f(\xi)\int_I k(\xi,\eta)\mbox{ }d\eta d\xi -\int_I \frac{\gamma(\xi)}{\overline k(\xi)}\int_I k(\xi,\eta)f(\eta)\mbox{ }d\eta d\xi \\
&=\int_I \gamma(\xi)f(\xi)\mbox{ }d\xi-\int_{I} f(\eta)\left( \int_{I} \gamma(\xi)\frac{k(\xi,\eta)}{\overline k(\xi)}\mbox{ }d\xi \right) \mbox{ }d\eta =0.
\end{align*}
\end{proof}

\section{Discount approximation}\label{discount}
In this Section, We consider \eqref{DP} under the assumptions: 



\begin{enumerate}[label=(B\arabic*)]

\item
There exist constants $C_1, C_2 >0$ and $m>1$ such that for all $(x,p,\xi)\in \Tt^d \times \Rr^d \times I$,
\[c_1 |p|^m -C_2 \leq H(x,p,\xi).\]
\label{coercive}

\item For each $R>0$, there exists a modulus of continuity $\omega_R$ such that
\[ |H(x,p,\xi)-H(y,p,\xi)|\leq \omega_R(|x-y|), \]
for all $x,y \in \Tt^d$, $p\in B(0,R)$ and $\xi \in I$.
\label{growth}

\end{enumerate}

As in \cite[Section 4]{Ishiishimano}, we can see that \eqref{DP} has the unique viscosity solution $v^\alpha \in C(\Tt^d) \otimes \mathcal{B}(I)$. Here, we give a proof of a comparison principle for \eqref{DP}.

\begin{pro}\label{new-comparison}
Suppose that Assumptions \ref{coercive} and \ref{growth} hold. Let $u\in \mathcal{U}^+(\Tt^d \times I)$ and $v\in \mathcal{U}^-(\Tt^d\times I)$ be a viscosity subsolution and a viscosity supersolution to \eqref{DP}, respectively. Assume that both $u$ and $-v$ are bounded above on $\Tt^d \times I$. Then, $u\leq v$ in $\Tt^d \times I$. 
\end{pro}

\begin{proof}
Suppose that $\sup_{\Tt^d \times I} u(x,\xi)-v(x,\xi)=:\lambda>0$. Let $\epsi>0$ be so small and set
\begin{equation*}\label{test}
\Phi(x,y,\xi):=u(x,\xi)-v(y,\xi)-\frac{|x-y|^2}{2\epsi}.
\end{equation*}
Then, there exists $( \hat x_\epsi, \hat y_\epsi, \hat \xi_\epsi) \in \Tt^d \times \Tt^d \times I$ such that 
\[ \sup_{\Tt^d \times \Tt^d \times I}\Phi-\epsi<\Phi(\hat x_\epsi, \hat y_\epsi, \hat \xi_\epsi).\]
Moreover, we can take $(x_\epsi,y_\epsi)\in \Tt^d\times \Tt^d$ satisfying $\Phi(x_\epsi,y_\epsi,\hat \xi_\epsi)=\max_{\Tt^d\times \Tt^d}\Phi(x,y,\hat \xi_\epsi)$. Choose their subsequences such that $x_\epsi \to \hat x$ and $y_\epsi \to \hat y$ for some $\hat x ,\hat y\in \Tt^d$.
 Noting that  
\begin{equation*}\label{eeee}
\lambda-\epsi<\Phi(x_\epsi,y_\epsi,\hat \xi_\epsi),
\end{equation*}
we have
\[\frac{|x_\epsi-y_\epsi|^2}{2\epsi} \leq \sup_{\Tt^d \times I} u + \sup_{\Tt^d\times I} (-v) +\epsi-\lambda. \]
Hence, we can see $\hat x=\hat y$.
Moreover, because $\frac{|x_\epsi-y_\epsi|^2}{2\epsi}\geq0$, we have
\begin{equation}\label{dddd}
\lambda-\epsi \leq u(x_\epsi,\hat \xi_\epsi)-v(y_\epsi.\hat \xi_\epsi).
\end{equation}

On the other hand, in view of the definition of viscosity solutions, we get
\begin{equation}\label{subbb}
\alpha u(x_\epsi,\hat \xi_\epsi)+H(x_\epsi,\frac{x_\epsi-y_\epsi}{\epsi},\hat \xi_\epsi)+\theta u(x_\epsi,\hat \xi_\epsi)\leq 0
\end{equation}
and
\begin{equation}\label{suppp}
\alpha v(y_\epsi,\hat \xi_\epsi)+H(y_\epsi,\frac{x_\epsi-y_\epsi}{\epsi},\hat \xi_\epsi)+\theta v(y_\epsi, \hat \xi_\epsi)\geq 0.
\end{equation}
Due to Assumption \ref{coercive}, \eqref{subbb} implies $|\frac{x_\epsi-y_\epsi}{\epsi}|< \hat R$ with some constant $\hat R>0$. 
Subtract \eqref{subbb} and \eqref{suppp} each other, to yield
\begin{align}\label{cccc}
\alpha\{ u(x_\epsi, \hat \xi_\epsi)-v(y_\epsi,\hat \xi_\epsi)\}
&\leq H(y_\epsi,\frac{x_\epsi-y_\epsi}{\epsi},\hat \xi_\epsi)-H(x_\epsi,\frac{x_\epsi-y_\epsi}{\epsi}, \hat \xi_\epsi)-\theta u(x_\epsi,\hat \xi_\epsi)+\theta v(y_\epsi, \hat \xi_\epsi) \notag \\
&\leq \omega_{\hat R} (|x_\epsi-y_\epsi|) -\theta u(x_\epsi,\hat \xi_\epsi)+\theta v(y_\epsi, \hat \xi_\epsi).
\end{align}
Because $\sup_{\Tt^d\times \Tt^d \times I} \Phi-\epsi<\Phi(x_\epsi,y_\epsi,\hat \xi_\epsi)$, we have

\begin{align*}
\theta u(x_\epsi,\hat \xi_\epsi)-\theta v(y_\epsi, \hat \xi_\epsi)
&\geq \int_I k(\hat \xi_\epsi,\eta)\{ \Phi(x_\epsi,y_\epsi,\hat \xi_\epsi)-\Phi(x_\epsi,y_\epsi,\eta)\} \mbox{ }d\eta \\
&\geq-\epsi \int_I k(\hat \xi_\epsi,\eta) \mbox{ }d\eta \geq -\epsi k_1.
\end{align*}
By \eqref{cccc}, we have 
\begin{equation*}
u(x_\epsi,\hat \xi_\epsi)-v(y_\epsi.\hat \xi_\epsi) \leq o_\epsi(1).
\end{equation*}
In view of \eqref{dddd}, we observe
$\lambda \leq 0$, which is a contradiction.
\end{proof}


In this Section, we obtain  uniform Lipschitz estimates for $v^\alpha$ with respect to $x\in \Tt^d$.
We first give an estimate on the gradient of $v^\alpha$ which may depend on $\alpha$.

\begin{lem} Suppose that Assumptions \ref{coercive} and \ref{growth} hold. Let $v^\alpha$ be the viscosity solution to \eqref{DP}.
Then, for all $\alpha>0$, there exist a constant $L_\alpha >0$ such that for $x,y \in \Tt^d$ and $\xi \in I$,
\begin{equation*}
|v^\alpha(x,\xi)-v^\alpha(y,\xi)|\leq L_\alpha|x-y|.
\end{equation*} 
\end{lem}
\begin{proof}
Let $M:= \sup_{(x,\xi) \in \Tt^d\times I}| H(x,0,\xi)|$. Then $-M/\alpha$ and $M/\alpha$ are a subsolution and a supersolution to \eqref{DP}, respectively. In the process of constructing the viscosity solution $v^\alpha$ to \eqref{DP}, we use Perron's method (see \cite[Section 4]{Ishiishimano} for details). Then, we have $ -M/\alpha \leq v^\alpha \leq M/\alpha$ in $\Tt^d \times I$. Thus, it holds
\begin{align*}
H(x,Dv^\alpha,\xi)
&\leq {M}-\int_I k(\xi,\eta)(v^\alpha(x,\xi)-v^\alpha(x,\eta))\mbox{ }d\eta \leq {M}+\frac{2k_1M}{\alpha}
\end{align*}
in the sense of viscosity solutions. In light of \ref{coercive}, we get the conclusion.
\end{proof}

To obtain the uniform estimate, we handle the coupling term $\theta v^\alpha$ as follows. 

\begin{lem}
Suppose that Assumptions \ref{coercive} and \ref{growth} hold. Let $v^\alpha$ be the viscosity solution to \eqref{DP}.
There exist constants $M_1, M_2 >0$ such that for all $\alpha>0$, $x, y \in \Tt^d$ and $\xi \in I$, we have
\begin{equation}\label{kesti} 
\left|\int_{I}  k(\xi,\eta)(v^\alpha(x,\xi)-v^\alpha(y,\eta))\mbox{ }d\eta \right| \leq M_1+M_2 \|D v^\alpha\|_{L^\infty(\Tt^d\times I)}.
\end{equation}
\end{lem}

\begin{proof}
Fix $\xi \in I$. Take $x_\xi \in \Tt^d$ as $\min_{x\in \Tt^d} v^\alpha (x,\xi)=v^\alpha(x_\xi,\xi)$. Then, by the definition of supersolutions, we get
\begin{equation}\label{ksup}
\int_I k(\xi,\eta)\{ v^\alpha(x_\xi,\xi)-v^\alpha(x_\xi,\eta)\} \mbox{ }d\eta \geq -M-H(x_\xi,0,\xi) \geq -M_1.
\end{equation}

 Next, for all $x,y\in \Tt^d$ and $\eta \in I$, we have
\begin{align*}
v^\alpha(x,\xi)-v^\alpha(y,\eta) 
&\geq v^\alpha(x_\xi, \xi)-v^\alpha(y,\eta)\\
&= v^\alpha(x_\xi, \xi)-v^\alpha(x_\xi,\eta)+v^\alpha(x_\xi,\eta)-v^\alpha(y,\eta)\\
&\geq  v^\alpha(x_\xi, \xi)-v^\alpha(x_\xi,\eta)-\|Dv^\alpha \|_{L^\infty(\Tt^d\times I)}.
\end{align*}
Multiply the above by $k(\xi,\eta)$ and integrate over $ I$, to get
\begin{align*}
\int_{I}  k(\xi,\eta)(v^\alpha(x,\xi)-v^\alpha(y,\eta))\mbox{ }d\eta 
\geq& \int_I k(\xi,\eta)\{ v^\alpha(x_\xi,\xi)-v^\alpha(x_\xi,\eta)\} \mbox{ }d\eta -k_1 \|Dv^\alpha \|_{L^\infty}.
\end{align*}
Using \eqref{ksup}, we get one side inequality of \eqref{kesti}. By a similar argument for the case of subsolutions, we get the opposite inequality.
\end{proof}

Finally, we obtain the uniform bound for $Dv^\alpha$ and $\theta v^\alpha$ in $\alpha>0$.

\begin{pro} Suppose that Assumptions \ref{coercive} and \ref{growth} hold. Let $v^\alpha$ be the viscosity solution to \eqref{DP}.
There exists a constant $C>0$ such that for all $\alpha>0$, 
\begin{equation}\label{esti}
 \|D v^\alpha \|_{L^\infty(\Tt^d \times I)}  +\|\Theta v^\alpha\|_{L^\infty(\Tt^d\times I)}  \leq C.
\end{equation}
\end{pro}

\begin{proof}
For all $\xi \in I$ and almost all $x\in \Tt^d$, it holds
\begin{align*}
H(x,Dv^\alpha(x,\xi),\xi)
&=-\alpha v^\alpha-\int_{I} k(\xi,\eta)\{ v^\alpha(x,\xi)-v^\alpha(x,\eta)\} \mbox{ }d\eta \\
& \leq M+M_1+M_2 \|Dv^\alpha \|_{L^\infty(\Tt^d\times I)},
\end{align*}
where we use \eqref{kesti}. In light of \ref{coercive}, we get 
\[ \|Dv^\alpha \|_{L^\infty(\Tt^d\times I)} \leq C. \]
 By \eqref{kesti} again, we get the conclusion.
\end{proof}

Here, we give a result on the stability of viscosity solutions for half-relaxed limits.

\begin{pro}\label{stability}
Suppose that Assumptions \ref{coercive} and \ref{growth} hold. Let $v^\alpha \in C(\Tt^d) \otimes \mathcal{B}(I)$ be the viscosity solution to \eqref{DP}. Additionally, suppose that there exist $c: I \to \Rr$ and $x_0 \in \Tt^d$ such that for all $\xi\in I$,
\begin{equation}\label{tech}
\lim_{\alpha \to 0}\alpha v^\alpha(x_0,\xi)+\Theta v^\alpha(x_0,\xi)=-c(\xi).
\end{equation}
 Let $\tilde v^\alpha (x,\xi):=v^\alpha(x,\xi)-v^\alpha(x_0,\xi)$. Define
\begin{equation}\label{url}
\overline v(x,\xi)=\lim_{r\to 0} \sup\{\tilde v^\alpha(y,\xi)\:\:|\:\: |x-y|+\alpha \leq r\},
\end{equation}
and
\begin{equation}\label{lrl}
\underline v(x,\xi)=\lim_{r\to 0} \inf\{\tilde v^\alpha(y,\xi)\:\:|\:\: |x-y|+\alpha \leq r\}.
\end{equation}
Then, $\overline v$ and $\underline v$ are a subsolution and a supersolution, respectively,  to 
\begin{equation}\label{EEEEE}
H(x,Dv,\xi)+\int_I k(\xi,\eta)\{v(x,\xi)-v(x,\eta)\} \mbox{ }d\eta=c(\xi) \quad\rm{in}\  \Tt^d\times I.
\end{equation} 
\end{pro}
\begin{proof}
Here, we only prove that $\overline v$ is a subsolution to \eqref{EEEEE}. Take $\xi \in I$, $\phi \in C^1(\Tt^d)$ and $\hat x \in \Tt^d$ such that $\overline v(\cdot,\xi)-\phi(\cdot)$ attains a strict maximum at $\hat x\in \Tt^d$. Let $\{y_j\}_{j\in \Nn}$ and $\{\alpha_j\}_{j\in \Nn}$ such that $y_j \to \hat x$, $\alpha_j \to 0$ and $\tilde v^{\alpha_j}(y_j,\xi) \to \overline v(\hat x,\xi)$.

Let $x_j \in \Tt^d$ be a point satisfying that $\max_{x\in \Tt^d}\tilde v^{\alpha_j}(x,\xi)-\phi(x)= \tilde v^{\alpha_j}(x_j,\xi)-\phi(x_j)$. We can choose a subsequence satisfying $x_j \to \overline x$, for some $\overline x \in \Tt^d$. Then, we get
\begin{align*}
\overline v(\hat x,\xi)-\phi(\hat x) 
&\leq \liminf_{j\to \infty}\tilde v^{\alpha_j}(y_j,\xi)-\phi(y_j) \leq \liminf_{j\to \infty}\tilde v^{\alpha_j}(x_j,\xi)-\phi(x_j)\\
&\leq \limsup_{j\to \infty}\tilde v^{\alpha_j}(x_j,\xi)-\phi(\overline x) \leq \overline v(\overline x,\xi)-\phi(\overline x).
\end{align*}
Thus, we have $\hat x=\overline x$ and $\lim_{j\to \infty} \tilde v^{\alpha_j}(x_j,\xi)=\overline v(\overline x,\xi)$.

Because $v^{\alpha_j}$ is a subsolution to \eqref{DP}, we have 
\[\alpha_j \tilde v^{\alpha_j}(x_j,\xi) +H(x_j,D\phi(x_j),\xi)+\alpha_j v^{\alpha_j}(x_0,\xi)+\Theta v^{\alpha_j}(x_0.\xi)\leq -\Theta \tilde v^{\alpha_j}(x_j,\xi). \]
In light of \eqref{esti}, $\tilde v^\alpha$ is bounded on $\Tt^d \times I$. Due to \eqref{tech} and the definition of $\{x_j\}_{j\in \Nn}$ and $\{\alpha_j\}_{j\in \Nn}$, we have
\begin{equation*}
 \alpha_j \tilde v^{\alpha_j}(x_j,\xi) +H(x_j,D\phi(x_j),\xi)+\alpha_j v^{\alpha_j}(x_0,\xi)+\Theta v^{\alpha_j}(x_0.\xi) \to H(\hat x,D\phi(\hat x),\xi)-c(\xi).
\end{equation*}
 By Fatou's lemma, we obtain
\begin{align*}
\limsup_{j\to \infty} \int_I k(\xi,\eta) \tilde v^{\alpha_j} (x_j,\eta)\mbox{ }d\eta \leq \int_I \limsup_{j\to \infty} k(\xi,\eta) \tilde v^{\alpha_j}  (x_j,\eta)\mbox{ }d\eta \leq \int_I k(\xi,\eta) \overline v(\hat x,\eta) \mbox{ }d\eta.
\end{align*}
Accordingly, it follows that
\[H(\hat x, D \phi(\hat x), \xi)+\int_I k(\xi,\eta)\{\overline v(\hat x,\xi)-\overline v(\hat x,\eta)\} \mbox{ }d\eta \leq c(\xi).\] \end{proof}

\begin{remark}\label{rem}
It is worth emphasizing that if $I$ is a countable set, we can construct a subsequence of $\{v^\alpha\}_{\alpha>0}$ satisfying \eqref{tech} by using a diagonal argument for $\xi$, due to a priori estimate given by \eqref{esti}. 
On the other hand, if $I$ is an uncountable set, because we do not have the equi-continuity of $\{v^\alpha\}_{\alpha>0}$ (particularly with respect to $\xi$), we do not know whether $v^\alpha -\min_{\Tt^d} v^\alpha$ converges on $\Tt^d \times I$ even along subsequences. Thus, as far as the author knows, the existence of solutions to the ergodic problem \eqref{EP} is not established yet in a general setting. In this paper, we give it in a specific Hamiltonian of the form $H(x,p,\xi)=F(x,p,\xi)-f(x,\xi)$ as described in Introduction.
\end{remark}

\section{convergence of the solution}\label{conv}
In what follows, we consider the Hamiltonian of the form \eqref{form}. Then, \eqref{DP}  and \eqref{EP} become
\begin{equation}\label{D}
\alpha v^\alpha(x,\xi)+F(x,Dv^\alpha(x,\xi),\xi)+\int_Ik(\xi,\eta)(v^\alpha(x,\xi)-v^\alpha(x,\eta))\mbox{ }d\eta=f(x,\xi) \quad\rm{in}\  \Tt^d\times I,
\end{equation}
and
\begin{equation}\label{E}
F(x,Dv(x,\xi),\xi)+\int_Ik(\xi,\eta)(v(x,\xi)-v(x,\eta))\mbox{ }d\eta=f(x,\xi)+c \quad\rm{in}\  \Tt^d\times I ,
\end{equation}
respectively. 
Firstly, we observe that $v^\alpha=0$ on $\mathcal{Z}$, defined in \eqref{uniquenessset}.
\begin{pro}\label{zero}
Suppose that Assumptions \ref{Coercive} - \ref{growth1} hold. Let $v^\alpha$ be the viscosity solution to \eqref{D}. Then, for all $(x,\xi) \in \mathcal{Z}$ and $\alpha>0$, we have $v^\alpha (x,\xi)=0$.
\end{pro}
\begin{proof}

Let $v^\alpha$ be the viscosity solutions of \eqref{D}. For $\delta>0$, set 
\begin{equation*}
v^\alpha_\delta(x,\xi):=\gamma^\delta*v^\alpha(x,\xi)=\int_{B(0,\delta)} \gamma^\delta(y)v^\alpha(x-y,\xi)\mbox{ }dy,
\end{equation*}
 where $\gamma^\delta(y):=\delta^{-d}\gamma(\delta^{-1}y)$ for $y\in \Tt^d$ and $\gamma$ is a standard mollifier.
 In view of \eqref{esti}, there exists a constant $R>\|Dv^\alpha\|_{L^\infty(\Tt^d \times I)}$.
 Due to Assumption \ref{convex} - \ref{growth1} and Jensen's inequality, for all $(x,\xi)\in \Tt^d\times I$, we get 
 \begin{align*}
 0&\leq F(x,Dv^\alpha_\delta(x,\xi),\xi)\\
 &=F\left(x,\int_{B(0,\delta)}\gamma^\delta(y)Dv^\alpha(x-y,\xi)\mbox{ }dy,\xi\right)\\
 &\leq \int_{B(0,\delta)}F(x,Dv^\alpha(x-y,\xi),\xi)\gamma^\delta(y)\mbox{ }dy \\
 &\leq \int_{B(0,\delta)}F(x-y,Dv^\alpha(x-y,\xi),\xi)\gamma^\delta(y)\mbox{ }dy+\omega_R (\delta) \\
 &\leq \int_{B(0,\delta)}-\gamma^\delta(y)  \left\{ \alpha v^\alpha(x-y,\xi)+\Theta v^\alpha(x-y,\xi)-f(x-y,\xi)   \right\}  \mbox{ }dy +\omega_R(\delta).
 \end{align*}
Thanks to the Lipschitz estimate in \eqref{esti}, sending $\delta \to 0$, we observe that 
\begin{equation}\label{subpro}
\alpha v^\alpha(x,\xi)+\Theta v^\alpha(x,\xi) \leq f(x,\xi).
\end{equation}
Let $x\in \mathcal{ \tilde A}$, defined in \eqref{tildeA}. Let $\gamma(\xi)$ and $\overline k(\xi)$ be defined in Section \ref{gamma}. Multiply \eqref{subpro} by $\gamma(\xi)/\overline k(\xi)$, and integrate over $I$, to yield
\begin{equation*}
\int_I \alpha v^\alpha(x,\xi) \frac{\gamma(\xi)}{\overline k(\xi)}\mbox{ }d\xi+\int_I \frac{\gamma(\xi)}{\overline k(\xi)} \Theta v^\alpha(x,\xi)\mbox{ }d\xi \leq  \int_I \frac{\gamma(\xi)}{\overline k(\xi)} f(x,\xi) \mbox{ }d\xi=0.
\end{equation*}
 In light of \eqref{couple},
we have
\begin{equation*}
\int_I \alpha v^\alpha(x,\xi) \frac{\gamma(\xi)}{\overline k(\xi)} \mbox{ }d\xi \leq 0.
\end{equation*}
Because $0$ is a subsolution to \eqref{D}, due to Proposition \ref{new-comparison}, $v^\alpha$ is nonnegative. Moreover, because both $\gamma$ and $\overline k$ are positive, $v^\alpha(x,\xi)=0$ for almost all $\xi \in I$. By \eqref{subpro} again, for all $(x,\xi)\in \mathcal{Z}$, we get
\begin{equation*}
\big( \alpha+\int_I k(\xi,\eta)\mbox{ }d\eta \big)v^\alpha(x,\xi) \leq0,
\end{equation*}
which implies $v^\alpha(x,\xi)=0$ for all $(x,\xi) \in \mathcal{Z}$.
\end{proof}

Hence, $v^\alpha$ uniformly converges to $0$ in $\mathcal{Z}$. 
Next, we show a comparison principle for \eqref{E} in terms of  $\mathcal{Z}$.

\begin{teo}\label{SCR}
Suppose that Assumptions \ref{Coercive}-\ref{growth1} hold. Let $u\in \mathcal{U}^+(\Tt^d \times I)$ and $v\in \mathcal{U}^-(\Tt^d\times I)$ be a viscosity subsolution and a viscosity supersolution to \eqref{E} with $c=0$, respectively. Assume that both $u$ and $-v$ are bounded above on $\Tt^d \times I$. If $u\leq v$ in $\mathcal{Z}$,
then $u\leq v$ in $\Tt^d\times I$.
\end{teo}

\begin{proof}
Let $0< \mu <1$. Suppose that 
\[M_\mu:=\sup_{(x,\xi)\in I} \{\mu u(x,\xi)-v(x,\xi)\}>0.\]
For each $\epsi>0$, take $(x_\epsi,\xi_\epsi)\in \Tt^d \times I$ satisfying that
$M_\mu-\epsi \leq \mu u(x_\epsi,\xi_\epsi)-v(x_\epsi,\xi_\epsi)$.
Moreover, we can choose subsequences satisfying $x_\epsi \to x_0$ as $\epsi \to 0$ and set 
\[E:=\{\xi \in I \:\:| \:\: (\mu u-v)(x_0,\xi)=M_\mu \}.\]
Then, we consider the following three cases:
\begin{enumerate}[label=Case \arabic*:]
\item $x_0 \in \mathcal{\tilde A}$.

\item $x_0 \not \in \mathcal{\tilde A}$ and $|I-E|>0$.

\item $x_0 \not \in \mathcal{\tilde A}$ and $|I-E|=0$.
\end{enumerate}

Case 1. First, we consider the case $x_0\in \mathcal{\tilde A}$. Take $\delta>0$. Set
\[\Phi(x,y):= \mu u(x,\xi_\epsi)-v(y,\xi_\epsi)-\frac{|x-y|^2}{2\delta}-|x-x_\epsi|^2,\]
and its maximum is attained at $(x^\delta_\epsi, y^\delta_\epsi)$. Then, we have
\begin{equation}\label{rrr}
\Phi(x^\delta_\epsi,y^\delta_\epsi)\geq \Phi(x_\epsi,x_\epsi)\geq M_\mu-\epsi,
\end{equation}
in particular,
\begin{equation}\label{mmmm}
\frac{|x^\delta_\epsi-y^\delta_\epsi|^2}{2\delta}+|x^\delta_\epsi-x_\epsi|^2 \leq \mu u(x^\delta_\epsi,\xi_\epsi)-v(y^\delta_\epsi,\xi_\epsi)+\epsi-M_\mu.
\end{equation}
Because $u$ and $-v$ are bounded above, we can choose subsequences such that $x^\delta_\epsi \to \hat x_\epsi$ and $y^\delta_\epsi \to \hat x_\epsi$, for a common $\hat x_\epsi$ as $\delta \to 0$. Moreover, by \eqref{mmmm}, we get 
\begin{align*}
|\hat x_\epsi -x_\epsi|^2&\leq\limsup_{\delta \to 0} |x^\delta_\epsi-x_\epsi|^2 \\
&\leq \limsup_{\delta \to 0} \big( \mu u(x^\delta_\epsi,\xi_\epsi)-v(y^\delta_\epsi,\xi_\epsi) \big)+\epsi-M_\mu\\
&\leq \mu u(\hat x_\epsi,\xi_\epsi)-v(\hat x_\epsi,\xi_\epsi)+\epsi-M_\mu \leq \epsi,
\end{align*} 
which implies $\hat x_\epsi \to x_0$ as $\epsi \to 0$. 

We set $p^\delta_\epsi=\frac{x^\delta_\epsi-y^\delta_\epsi}{\delta}$. By the definition of viscosity solutions, we get

\begin{equation}\label{sub2}
\mu F(x^\delta_\epsi,\frac{p^\delta_\epsi+2(x^\delta_\epsi- x_\epsi)}{\mu}, \xi_\epsi)+\mu \Theta u(x^\delta_\epsi,\xi_\epsi)\leq \mu f(x^\delta_\epsi, \xi_\epsi),
\end{equation}
and
\begin{equation}\label{sup2}
F(y^\delta_\epsi,p^\delta_\epsi, \xi_\epsi)+\Theta v(y^\delta_\epsi,\xi_\epsi)\geq  f(y^\delta_\epsi, \xi_\epsi).
\end{equation}
In view of \ref{Coercive}, \eqref{sub2} implies $p^\delta_\epsi$ is bounded uniformly with respect to $\epsi$ and $\delta$. Subtracting \eqref{sub2} and \eqref{sup2} each other, to yield
\begin{equation}\label{gggg}
\mu F(x^\delta_\epsi,\frac{p^\delta_\epsi+2(x^\delta_\epsi-x_\epsi)}{\mu},  \xi_\epsi)-F(y^\delta_\epsi,p^\delta_\epsi, \xi_\epsi) +\mu \Theta u(x^\delta_\epsi,\xi_\epsi) -\Theta v(y^\delta_\epsi, \xi_\epsi) \leq \mu f(x^\delta_\epsi \xi_\epsi) - f(y^\delta_\epsi,\xi_\epsi).
\end{equation}
Let $R>\max\{|p^\delta_\epsi|,|\frac{2(x^\delta_\epsi-x_\epsi)}{\mu}|\}$. Due to \ref{convex} and \ref{growth1},  we have
\begin{align*}
&\mu F(x^\delta_\epsi,\frac{p^\delta_\epsi+2(x^\delta_\epsi-x_\epsi)}{\mu},  \xi_\epsi)-F(y^\delta_\epsi,p^\delta_\epsi, \xi_\epsi)\\
&\geq \mu F(x^\delta_\epsi,\frac{p^\delta_\epsi+2(x^\delta_\epsi-x_\epsi)}{\mu},  \xi_\epsi)-F(x^\delta_\epsi,p^\delta_\epsi, \xi_\epsi)-\omega_R(|x^\delta_\epsi-y^\delta_\epsi|)\\
&\geq -(1-\mu)F(x^\delta_\epsi,-\frac{2(x^\delta_\epsi-x_\epsi)}{1-\mu},\xi_\epsi)-\omega_R(|x^\delta_\epsi-y^\delta_\epsi|).
\end{align*}
 In light of \ref{positive} and \ref{growth1}, we get
\begin{align*}
&-(1-\mu)F(x^\delta_\epsi,-\frac{2(x^\delta_\epsi-x_\epsi)}{1-\mu},\xi_\epsi)-\omega_R(|x^\delta_\epsi-y^\delta_\epsi|) \\
&\geq -(1-\mu)F(x^\delta_\epsi,0,\xi_\epsi) -(1-\mu)\omega_R\left(\left|\frac{2(x^\delta_\epsi-x_\epsi)}{1-\mu}\right|\right)-\omega_R(|x^\delta_\epsi-y^\delta_\epsi|)\\
&=-(1-\mu)\omega_R\left(\left|\frac{2(x^\delta_\epsi-x_\epsi)}{1-\mu}\right|\right)-\omega_R(|x^\delta_\epsi-y^\delta_\epsi|).
\end{align*}
Thus, by \eqref{gggg}, we obtain
\begin{align}\label{aaaa}
& \mu \Theta u(x^\delta_\epsi, \xi_\epsi)-\Theta v(y^\delta_\epsi,\xi_\epsi)\notag \\
&\leq \mu f(x^\delta_\epsi,\xi_\epsi) -f(y^\delta_\epsi,\xi_\epsi)+(1-\mu)\omega_R\left(\left|\frac{2(x^\delta_\epsi-x_\epsi)}{1-\mu}\right|\right)+\omega_R(|x^\delta_\epsi-y^\delta_\epsi|)\notag \\
&\leq (\mu-1) f(x^\delta_\epsi,\xi_\epsi)+\omega(|x^\delta_\epsi-y^\delta_\epsi|)+(1-\mu)\omega_R\left(\left|\frac{2(x^\delta_\epsi-x_\epsi)}{1-\mu}\right|\right)+\omega_R(|x^\delta_\epsi-y^\delta_\epsi|)\notag \\
&\leq \omega(|x^\delta_\epsi-y^\delta_\epsi|)+(1-\mu)\omega_R\left(\left|\frac{2(x^\delta_\epsi-x_\epsi)}{1-\mu}\right|\right)+\omega_R(|x^\delta_\epsi-y^\delta_\epsi|) .
\end{align} 

On the other hand, using Fatou's lemma, we have
\begin{align}\label{ppppp}
\limsup_{\delta \to 0} \int_I k(\xi_\epsi,\eta) \big(\mu u(x^\delta_\epsi,\eta) -v(y^\delta_\epsi,\eta) \big)\mbox{ }d\eta 
&\leq \int_I k( \xi_\epsi,\eta)  \big( \mu u(\hat x_\epsi,\eta)-v(\hat x_\epsi,\eta) \big) \mbox{ }d\eta,
\end{align}
and by \eqref{rrr},
\begin{equation}\label{qqqqq}
\int_I k(\xi_\epsi,\eta) \big(\mu u(x^\delta_\epsi,\xi_\epsi) -v(y^\delta_\epsi,\xi_\epsi) \big)\mbox{ }d\eta 
\geq (M_\mu-\epsi) \int_I k(\xi_\epsi,\eta) \mbox{ }d\eta.
\end{equation}
Combine \eqref{ppppp} and \eqref{qqqqq}, to yield
\begin{align*}
\mu \Theta u(x^\delta_\epsi, \xi_\epsi)-\Theta v(y^\delta_\epsi,\xi_\epsi)
\geq  k_0 \int_I M_\mu- \big( \mu u(\hat x_\epsi,\eta)-v(\hat x_\epsi,\eta) \big) \mbox{ }d\eta-\epsi k_1-o_\delta(1).
\end{align*}
By \eqref{aaaa}, sending $\delta \to 0$, we get
\begin{align*}
(1-\mu)\omega_R\left(\left|\frac{2(\hat x_\epsi-x_\epsi)}{1-\mu}\right|\right) \geq  k_0 \int_I M_\mu- \big( \mu u(\hat x_\epsi,\eta)-v(\hat x_\epsi,\eta) \big) \mbox{ }d\eta-\epsi k_1.
\end{align*}
Using Fatou's lemma, as $\epsi \to 0$, we obtain
\begin{align}\label{kkkkk}
0\geq \int_I M_\mu - \big( \mu u( x_0,\eta)-v(x_0,\eta) \big) \mbox{ }d\eta.
\end{align}
It follows from $x_0 \in \mathcal{\tilde A}$ and $u\leq v$ in $\mathcal{Z}$ that  $u(x_0,\xi)\leq v(x_0,\xi)$ for almost all $\xi \in I$. Thus, we obtain
\begin{align*}
M_\mu \leq (1-\mu) \int_I -v(x_0,\eta)\mbox{ }d\eta,
\end{align*}
which contradicts $M_1>0$.


Case 2. Next, we consider the case $x_0 \not \in \mathcal{\tilde A}$ and $|I-E| > 0$. Then, there exists $\hat E \subset E^c$ satisfying $|\hat E|>0$ and 
\[ \sup_{\xi \in \hat E}\{\mu u(x_0,\xi)-v(x_0,\xi)\}<M_\mu.\] 
Set 
\[\rho := M_\mu-\sup_{\xi \in \hat E}\{\mu u(x_0,\xi)-v(x_0,\xi)\}.\]

Repeating the same argument in the Case 1, we have \eqref{kkkkk}. Then, we observe
\begin{align*}
0\geq \int_I M_\mu - \big( \mu u( x_0,\eta)-v(x_0,\eta) \big) \mbox{ }d\eta \geq \rho|\hat E|,
\end{align*}
which is a contradiction.


Case 3. We consider the case $x_0 \not \in \mathcal{\tilde A}$ and $|I-E|=0$. Then, there exists $J \subset I$ such that $|J|>0$ and $f(x_0,\xi)>0$ for all $\xi \in J$. Thus, there exists $\hat \xi \in E \cap J$.
Take $\beta>0$ and consider 
\[\Psi(x,y)=\mu u(x,\hat \xi)-v(y,\hat \xi)-\frac{|x-y|^2}{2\beta}-|x-x_0|^2, \]
and its maximum is attained  at  $(x_\beta,y_\beta)$.
Note that $\Psi(x_\beta,y_\beta)\geq \Psi(x_0,x_0)=M_\mu$, in particular,
\begin{equation*}
\frac{|x_\beta-y_\beta|^2}{2\beta}+|x_\beta-x_0|^2 \leq \mu u(x_\beta,\hat \xi)-v(y_\beta,\hat \xi)-M_\mu.
\end{equation*}
Thus, we can choose subsequences such that 
\[x_\beta \to x_0,\quad y_\beta \to x_0, \quad \frac{|x_\beta-y_\beta|^2}{2\beta} \to 0.\]
We set $p_\beta=\frac{x_\beta-y_\beta}{\beta}$. By the definition of viscosity solutions, we get

\begin{equation}\label{sub3}
\mu F(x_\beta,\frac{p_\beta+2(x_\beta- x_0)}{\mu}, \hat \xi)+\mu \Theta u(x_\beta,\hat \xi)\leq \mu f(x_\beta,\hat \xi)
\end{equation}
and
\begin{equation}\label{sup3}
F(y_\beta,p_\beta, \hat \xi)+\Theta v(y_\beta,\hat \xi)\geq  f(y_\beta,\hat \xi).
\end{equation}
In view of Assumption \ref{Coercive}, \eqref{sub3} implies $p_\beta$ is bounded uniformly in $\beta$. Subtracting \eqref{sub3} and \eqref{sup3} each other, to yield
\begin{equation*}
\mu F(x_\beta,\frac{p_\beta+2(x_\beta-x_0)}{\mu}, \hat \xi)-F(y_\beta,p_\beta, \hat \xi) +\mu \Theta u(x_\beta,\hat \xi) -\Theta v(y_\beta,\hat \xi) \leq \mu f(x_\beta,\hat \xi) - f(y_\beta,\hat \xi).
\end{equation*}
Due to Assumptions \ref{convex}-\ref{growth1}, as $\beta \to 0$, we have
\begin{align*}
\mu F(x_\beta,\frac{p_\beta+2(x_\beta-x_0)}{\mu}, \hat \xi)-F(x_\beta,p_\beta, \hat \xi)
&\geq -(1-\mu)F(x_\beta,-\frac{2(x_\beta-x_0)}{1-\mu},\hat \xi)\\
& \to -(1-\mu)F(x_0, 0, \hat \xi)=0.
\end{align*}
Thus, we get
\begin{align}\label{AAAA}
 \mu \Theta u(x_\beta,\hat \xi)-\Theta v(y_\beta,\hat \xi) 
&\leq \mu f(x_\beta,\hat \xi) -f(y_\beta,\hat \xi)+o_\beta(1)\notag \\
&\leq (\mu-1) f(x_0,\hat \xi)+o_\beta(1) .
\end{align}
On the other hand, using Fatou's lemma, we get
\begin{align*}
\limsup_{\beta \to 0} \int_I k(\hat \xi,\eta) \big(\mu u(x_\beta,\eta) -v(y_\beta,\eta) \big)\mbox{ }d\eta 
&\leq \int_I k(\hat \xi,\eta)  \big( \mu u(x_0,\eta)-v(x_0,\eta) \big) \mbox{ }d\eta  \\
&\leq M_\mu  \int_I k(\hat \xi,\eta)\mbox{ }d\eta.
\end{align*}
Moreover, because $\Psi (x_\beta,y_\beta)\geq M_\mu$, it holds
\[\int_I k(\hat \xi,\eta)\{ u(x_\beta,\hat \xi)-v(y_\beta,\hat \xi)\}\mbox{ }d\eta \geq M_\mu  \int_I k(\hat \xi,\eta)\mbox{ }d\eta .\]
Thus, we obtain 
\[-o_\beta(1)\leq  \mu \Theta u(x_\beta,\hat \xi)-\Theta v(y_\beta,\hat \xi), \]
 which contradicts \eqref{AAAA}.
\end{proof}

Finally, we prove the main result in this paper.

\begin{proof}[Proof of Theorem \ref{result1}]
Take $x_0 \in \mathcal{A}$. 
Define $\overline v$ and $\underline v$ as \eqref{url} and \eqref{lrl}, respectively.  In light of Proposition \ref{zero}, for all $\xi \in I$ and $\alpha>0$, we get
\begin{equation*}
\alpha v^\alpha(x_0,\xi)+\Theta v^\alpha(x_0,\xi)=0,
\end{equation*}  which tells us that \eqref{tech} holds with $c(\xi)\equiv0$. By Proposition \ref{stability}, $\overline v$ and $\underline v$ are a subsolution and a supersolution to \eqref{E} with $c=0$, respectively.

By the definition, we see $\overline v \geq \underline v$ in $\Tt^d\times I$.
On the other hand, Proposition \ref{zero} and \eqref{esti} imply $\overline v = \underline v=0$ in $\mathcal{Z}$. Then, Theorem \ref{SCR} guarantees that $\overline v \leq \underline v$ in $\Tt^d\times I$. Hence, $\overline v = \underline v=:v$ in $\Tt^d\times I$, and $v\in C(\Tt^d)\otimes \mathcal{B}(I)$.
By the basic property of half-relaxed limits, for each $\xi \in I$, we can see $v^\alpha(\cdot,\xi)\to v(\cdot, \xi)$ uniformly in $\Tt^d$ as $\alpha \to 0$. 
\end{proof}

\section{appendix}

Here, we study the discount approximation with the conditions that  $F\in C(\Tt^d \times \Rr^d \times I)$, $f\in C(\Tt^d \times I)$ and $k\in C(I\times I)$.  Then, as in \cite[Theorem 4.5]{Ishiishimano}, the unique viscosity solution $v^\alpha$ to \eqref{DP} belongs to $C(\Tt^d\times I)$. The aim of this Appendix is to prove the following refined asymptotics.

\begin{theorem}\label{result2}
Let $I\subset \Rr$ be closed and finite interval. Suppose that \ref{Coercive}-\ref{growth1} hold. Assume that $F\in C(\Tt^d \times \Rr^d \times I)$, $f\in C(\Tt^d \times I)$ and $k\in C(I\times I)$. Let $v^\alpha \in C(\Tt^d \times I)$ be the viscosity solution to \eqref{DP}. Then, there exists $v\in C(\Tt^d \times I)$, which is a viscosity solution to \eqref{EP} with $c=0$, satisfying that $v^\alpha \to v$ uniformly on $\Tt^d \times I$ as $\alpha \to 0$.

\end{theorem}

To prove the above, we consider other half-relaxed limits of $v^\alpha$. 
Instead of Proposition \ref{stability}, we give another result on the stability of viscosity solutions.

\begin{pro}\label{stability2}
Suppose that Assumptions \ref{coercive} and \ref{growth} hold.  Assume that $H\in C(\Tt^d \times \Rr^d \times I)$ and $k\in C(I\times I)$. Let $v^\alpha \in C(\Tt^d \times I)$ be the viscosity solution to \eqref{DP}. Additionally, suppose that there exist $c: I\to \Rr$ and $x_0 \in \Tt^d$ such that uniformly in $\xi\in I$,
\begin{equation}\label{app1}
\lim_{\alpha \to 0}\alpha v^\alpha(x_0,\xi)+\Theta v^\alpha(x_0,\xi)=-c(\xi).
\end{equation}
Let $\tilde v^\alpha (x,\xi):=v^\alpha(x,\xi)-v^\alpha(x_0,\xi)$ and set
\begin{equation}\label{app2}
\overline v(x,\xi)=\lim_{r\to 0} \sup\{\tilde v^\alpha(y,\eta)\:\:|\:\: |x-y|+|\xi-\eta|+\alpha \leq r\},
\end{equation}
and
\begin{equation}\label{app3}
\underline v(x,\xi)=\lim_{r\to 0} \inf\{\tilde v^\alpha(y,\eta)\:\:|\:\: |x-y|+|\xi-\eta|+\alpha \leq r\}.
\end{equation}
 Then, $\overline v$ and $\underline v$ are a subsolution and a supersolution, respectively,  to 
\begin{equation*}
H(x,Dv,\xi)+\int_I k(\xi,\eta)\{v(x,\xi)-v(x,\eta)\} \mbox{ }d\eta=c \quad\rm{in}\  \Tt^d\times I.
\end{equation*} 
\end{pro}

\begin{proof}
Here, we only show that $\overline v$ is a subsolution. Take $\xi \in I$, $\phi \in C^1(\Tt^d)$ and $\hat x \in \Tt^d$ such that $\overline v(\cdot,\xi)-\phi(\cdot)$ attains a strict maximum at $x_0\in \Tt^d$. Let $\{y_j\}_{j\in \Nn}$, $\{\xi_j\}_{j\in \Nn}$ and $\{\alpha_j\}_{j\in \Nn}$ such that $y_j \to \hat x$, $\xi_j \to \xi$, $\alpha_j \to 0$ and $\tilde v^{\alpha_j}(y_j,\xi_j) \to \overline v(\hat x,\xi)$.

On the other hand, let $x_j \in \Tt^d$ such that $\max_{x\in \Tt^d}\tilde v^{\alpha_j}(x,\xi_j)-\phi(x)= \tilde v^{\alpha_j}(x_j,\xi_j)-\phi(x_j)$. Moreover, we can choose a subsequence satisfying $x_j \to \overline x$, for some $\overline x \in \Tt^d$. Then, we get
\begin{align*}
\overline v(\hat x,\xi)-\phi(\hat x) 
&\leq \liminf_{j\to \infty}\tilde v^{\alpha_j}(y_j,\xi_j)-\phi(y_j) \leq \liminf_{j\to \infty}\tilde v^{\alpha_j}(x_j,\xi_j)-\phi(x_j)\\
&\leq \limsup_{j\to \infty}\tilde v^{\alpha_j}(x_j,\xi_j)-\phi(\overline x) \leq \overline v(\overline x,\xi)-\phi(\overline x).
\end{align*}
Thus, we have $\hat x=\overline x$ and $\lim_{j\to \infty} \tilde v^{\alpha_j}(x_j,\xi_j)=\overline v(\overline x,\xi)$.

Because $v^{\alpha_j}$ is a subsolution to \eqref{DP}, we have 
\[\alpha_j \tilde v^{\alpha_j}(x_j,\xi_j) +H(x_j,D\phi(x_j),\xi_j)+\alpha_j v^{\alpha_j}(x_0,\xi_j)+\Theta v^{\alpha_j}(x_0.\xi_j)\leq -\Theta \tilde v^{\alpha_j}(x_j,\xi_j). \]
In light of \eqref{esti}, $\tilde v^\alpha$ is bounded on $\Tt^d \times I$. 
Due to \eqref{app1}, we have
\begin{equation*}
 \alpha_j \tilde v^{\alpha_j}(x_j,\xi_j) +H(x_j,D\phi(x_j),\xi_j)+\alpha_j v^{\alpha_j}(x_0,\xi_j)+\Theta v^{\alpha_j}(x_0.\xi_j) \to H(\hat x,D\phi(\hat x),\xi)-c(\xi).
\end{equation*}
By Fatou's lemma, we obtain
\begin{align*}
\limsup_{j\to \infty} \int_I k(\xi_j,\eta) \tilde v^{\alpha_j} (x_j,\eta)\mbox{ }d\eta \leq \int_I \limsup_{j\to \infty} k(\xi_j,\eta) \tilde v^{\alpha_j}  (x_j,\eta)\mbox{ }d\eta \leq \int_I k(\xi,\eta) \overline v(\hat x,\eta) \mbox{ }d\eta.
\end{align*}
Hence, it follows that
\begin{equation*}
H(\hat x,D\phi(\hat x),\xi)+\int_I k(\xi,\eta)\{\overline v(\hat x,\xi)-\overline v(\hat x,\eta)\} \mbox{ }d\eta \leq c(\xi).
\end{equation*} 
\end{proof}

\begin{proof}[Proof of Theorem \ref{result2}.]
Note that $\mathcal{A}=\mathcal{\tilde A}$ and $\mathcal{Z}=\mathcal{A}\times I$ because $f\in C(\Tt^d\times I)$. Take $x_0 \in \mathcal{A}$. 
Define $\overline v$ and $\underline v$ as \eqref{app2} and \eqref{app3}, respectively.  In light of Proposition \ref{zero}, for all $\xi \in I$ and $\alpha>0$, we get
\begin{equation*}
\alpha v^\alpha(x_0,\xi)+\Theta v^\alpha(x_0,\xi)=0,
\end{equation*}  which tells us that \eqref{app1} holds with $c(\xi)\equiv0$. By Proposition \ref{stability2}, $\overline v$ and $\underline v$ are a subsolution and a supersolution to \eqref{E} with $c=0$, respectively.

By the definition, we see $\overline v \geq \underline v$ in $\Tt^d\times I$.
On the other hand, Proposition \ref{zero} and \eqref{esti} imply $\overline v = \underline v=0$ in $\mathcal{A}\times I$. Then, Theorem \ref{SCR} guarantees $\overline v \leq \underline v$ in $\Tt^d\times I$. Hence, $\overline v = \underline v=:v$ in $\Tt^d\times I$. 
Because $\overline v$ and $-\underline v$ are upper semi-continuous in $\Tt^d \times I$, $v$ is continuous on $\Tt^d\times I$.
Noting that $\Tt^d\times I$ is compact, by the basic property of half-relaxed limits, we can see $v^\alpha \to v$ uniformly on $\Tt^d \times I$ as $\alpha \to 0$. 
\end{proof}

\vspace{2mm}
{\bf Acknowledgement.} The author would like to thank Professor Hiroyoshi Mitake for his helpful comments and suggestions.


\begin{thebibliography}{100}



\bibitem{cami2} F. Camilli, O. Ley, P. Loreti, and V. D. Nguyen, {\it Large time behavior of weakly coupled systems of first-order Hamilton-Jacobi equations}, Nonlinear Differ. Equ. Appl. 19, 719-749, 2012.

\bibitem{davifathi} A. Davini, A. Fathi, R. Iturriaga and M. Zavidovique, {\it Convergence of the solutions of the discounted Hamilton-Jacobi equation,} Invent. Math. 206(1), 29-55, 2016. 



\bibitem{davizavi}  A. Davini and M. Zavidovique, {\it Convergence of the solutions of discounted Hamilton-Jacobi systems} Adv. Calc. Var., DOI: 10.1515/acv-2018-0037,  2019.









\bibitem{Enger} H. Engler and S. M. Lenhart, {\it Viscosity solutions for weakly coupled systems of Hamilton-Jacobi equations,} Proc. London Math. Soc. (3) 63(1), 212-240, 1991.



\bibitem{FS} W. H. Fleming and H. M. Soner, {\it Controlled Markov process and viscosity solutions,} Springer, New York, 1993.




\bibitem{Ishii2} H. Ishii, {\it The vanishing discount problem for monotone systems of Hamilton-Jacobi equations. Part 1: linear coupling}, Mathematics in Engineering 3(4), DOI:10.3934/mine.2021032, 2020.

\bibitem{Ishii4} H. Ishii, {\it An example in the vanishing discount problem for monotone systems of Hamilton-Jacobi equations}, Preprint in arXiv: 2006.02769, 2020.

\bibitem{Ishii3} H. Ishii and L. Jin, {\it The vanishing discount problem for monotone systems of Hamilton-Jacobi equations. Part 2: nonlinear coupling}, Calc. Var. Partial Differential Equations 59, Article number 140, 2020.

\bibitem{Koike} H. Ishii and S. Koike, {\it Viscosity solutions for monotone systems of second-order elliptic PDEs,} Comm. Partial Differential Equations 16 (6-7), 1095-1128, 1991.



\bibitem{IMT} H. Ishii, H. Mitake and H. V. Tran, {\it The vanishing discount problem and viscosity Mather measures. Part 1: The problem on a torus,} J.  Math. Pures Appl. (9) 108(2), 125-149, 2017. 

\bibitem{Ishiishimano}H. Ishii and  K. Shimano, {\it Asymptotic analysis for a class of infinite systems of first-order PDE: nonlinear parabolic PDE in the singular limit,} Comm. Partial Differential Equation 28(1/2), 409-438, 2003.


\bibitem{LPV} P. -L. Lions, G. Papanicolaou and S. R. S.  Varadhan, {\it Homogenization of Hamilton-Jacobi equations,} unpublished work, 1987. 





\bibitem{MT2} H. Mitake and H. V. Tran, {\it Remarks on the large-time behavior of viscosity solutions of quasi-monotone weakly coupled systems of Hamilton-Jacobi equations}, Asymptot. Anal. 77, 43-70, 2012.



\bibitem{MT5} H. Mitake and H. V. Tran, {\it Selection problems for a discount degenerate viscous Hamilton-Jacobi equation,} Adv. Math. 306, 684-703, 2017.


\bibitem{Namah} G. Namah and J. -M. Roquejoffre, {\it Remarks on the long time behaviour of the solutions of Hamilton-Jacobi equations,} Comm. Partial Differential Equation 24(5-6), 883-893, 1999.

\bibitem{Shimano} K. Shimano, {\it Homogenization and penalization of functional first-order PDE,} Nonlinear Differ. Equ. Appl. 13, 1-21, 2006.




\bibitem{YZ} G. G. Yin and Q. Zhang, {\it Continuous-time Markov chains and applications: A singular perturbation approach,} Springer-Verlag New York, 1998.

\bibitem{Ziliotto} B. Ziliotto, {\it Convergence of the solutions of the discounted Hamilton-Jacobi equation: a counterexample,} J. Math. Pures Appl. 128, 330-338, 2019.

\end{thebibliography}
\end{document}